\documentclass{article}

\usepackage{command_MU}
\usepackage{command_alexis}
\usepackage[font=small,labelfont=bf]{caption}
\usepackage{multirow}
\usepackage{booktabs}
\usepackage{hyperref}
\usepackage{enumitem}
\usepackage{authblk}

\usepackage{setspace}
\title{Stable Parametrization of Continuous and Piecewise-Linear Functions\thanks{This work was supported by the Swiss National Science Foundation, Grant 200020\_184646 / 1.}}

\author[1]{Alexis Goujon \thanks{alexis.goujon@epfl.ch}}
\author[1]{Joaquim Campos}%  \thanks{sjoaquim.campos@epfl.ch}}
\author[1]{Michael Unser}% \thanks{michael.unser@epfl.ch}}
\affil[1]{\'Ecole polytechnique f\'ed\'erale de Lausanne}

\begin{document}
\maketitle
% {\color{blue}
% \begin{itemize}
%   \item explain border functions,
%   \item in figure representation change formule to show number of hyperplane can be less than d
%   \item new view on linearboxsplines, review intro
%   \item conclusion
%   \item Outline
% \end{itemize}
% }
\begin{abstract}
  Rectified-linear-unit (ReLU) neural networks, which play a prominent role in deep learning, generate continuous and piecewise-linear (CPWL) functions. While they provide a powerful parametric representation, the mapping between the parameter and function spaces lacks stability. In this paper, we investigate an alternative representation of CPWL functions that relies on local hat basis functions. It is predicated on the fact that any CPWL function can be specified by a triangulation and its values at the grid points. We give the necessary and sufficient condition on the triangulation (in any number of dimensions) for the hat functions to form a Riesz basis, which ensures that the link between the parameters and the corresponding CPWL function is stable and unique. In addition, we provide an estimate of the $\ell_2\rightarrow L_2$ condition number of this local representation. Finally, as a special case of our framework, we focus on a systematic parametrization of $\R^d$ with control points placed on a uniform grid. In particular, we choose hat basis functions that are shifted replicas of a single linear box spline. In this setting, we prove that our general estimate of the condition number is optimal. We also relate our local representation to a nonlocal one based on shifts of a causal ReLU-like function.
\end{abstract}

%\linenumbers

\section{Introduction}
\subsection{Continuous and Piecewise-Linear Functions for Supervised Learning}
The purpose of supervised learning is to reconstruct an unknown mapping from a set of samples \cite{bishop2006pattern}. Namely, given a collection of training data pairs $(\M v_k, y_k) \in \mathbb{R}^{d}\times \mathbb{R}$ for $k=1,\ldots,K$, one wants to find $f \colon \mathbb{R}^{d} \rightarrow \mathbb{R}$ such that $f(\M v_k) \approx y_k$ for $k=1,\ldots,K$, without overfitting. As such, the problem is ill-posed. To make it numerically tractable, a reconstruction space $\mathcal{H}$ is chosen as the image of a finite-dimensional parameter space $\Theta$ under a given synthesis operator $T\colon \Theta \rightarrow\mathcal{H}$. This operator maps a parameter $\V \theta \in \Theta$ to its continuous representation $T\{\V \theta\}\in \mathcal{H}$. A celebrated way to choose the synthesis operator is to pick a feedforward neural network architecture. Given the multidimensional parameter $\V \theta = (\V \theta_1,\ldots,\V \theta_{L+1})\in \Theta$, we then have that
\begin{equation}
T\{\V \theta\} = (\V f_{\V \theta_{L+1} }\circ\V\sigma_L \circ \V f_{\V \theta_L} \circ  \V\sigma_{L-1} \circ \cdots \circ \V\sigma_2 \circ \V f_{\V \theta_2} \circ \V\sigma_1 \circ \V f_{\V \theta_1} ),
\end{equation}
where $L$ is the number of hidden layers of the neural network, $\V f_{\V \theta_k}\colon\R^{d_k}\rightarrow\R^{d_{k+1}}$ is an affine function parametrized by $\V \theta_k$, and $\V \sigma_k$ is an activation function that is chosen \textit{a priori}. For the model to be expressive, the activation functions need to be nonaffine---otherwise the generated function would remain trivially affine. Interestingly, the pointwise rectified-linear-unit (ReLU) function $x\mapsto\max(x,0)$, one of the simplest nonlinear functions, provides state-of-the-art performance \cite{Lecun2015,glorotDeepSparseRectifier2011}. In this case, $T\{\V \theta\}$ is the composition of continuous piecewise-linear (CPWL) functions, which turns out to be a CPWL function as well \cite{montufarNumberLinearRegions2014}. Remarkably, the reverse also holds true: any CPWL function $\mathbb{R}^d\rightarrow \mathbb{R}$ can be parametrized by a deep neural network with at most $\lceil \log_2(d+1) \rceil$ hidden layers \cite{Arora2018}. The depth of the architecture is instrumental to improve the approximation power of the network \cite{montufarNumberLinearRegions2014,eldanPowerDepthFeedforward2016,mhaskarDeepVsShallow2016} and its generalization ability \cite{poggioWhyWhenCan2017}, but it is a serious obstacle to the control of the model. For instance, to control the Lipschitz constant of a feedforward neural network, state-of-the-art techniques rely on theoretical upper bounds that worsen each time a new layer is added 
 \cite{Gouk2021,Scaman2018}. The depth is also detrimental to the interpretability of the parametrization: in deep networks, the effect of a parameter on the constructed mapping is poorly understood.

\subsection{Linear Expansion of Continuous and Piecewise-Linear Functions}
More interpretable representations of CPWL functions are provided by linear expansions. They boil down to two families: local and nonlocal representations (Figure \ref{fig:CPWL_Representation}).
\begin{figure}[h!]
  \begin{minipage}{1.0\linewidth}
    \centering
    \centerline{\includegraphics[width=120mm]{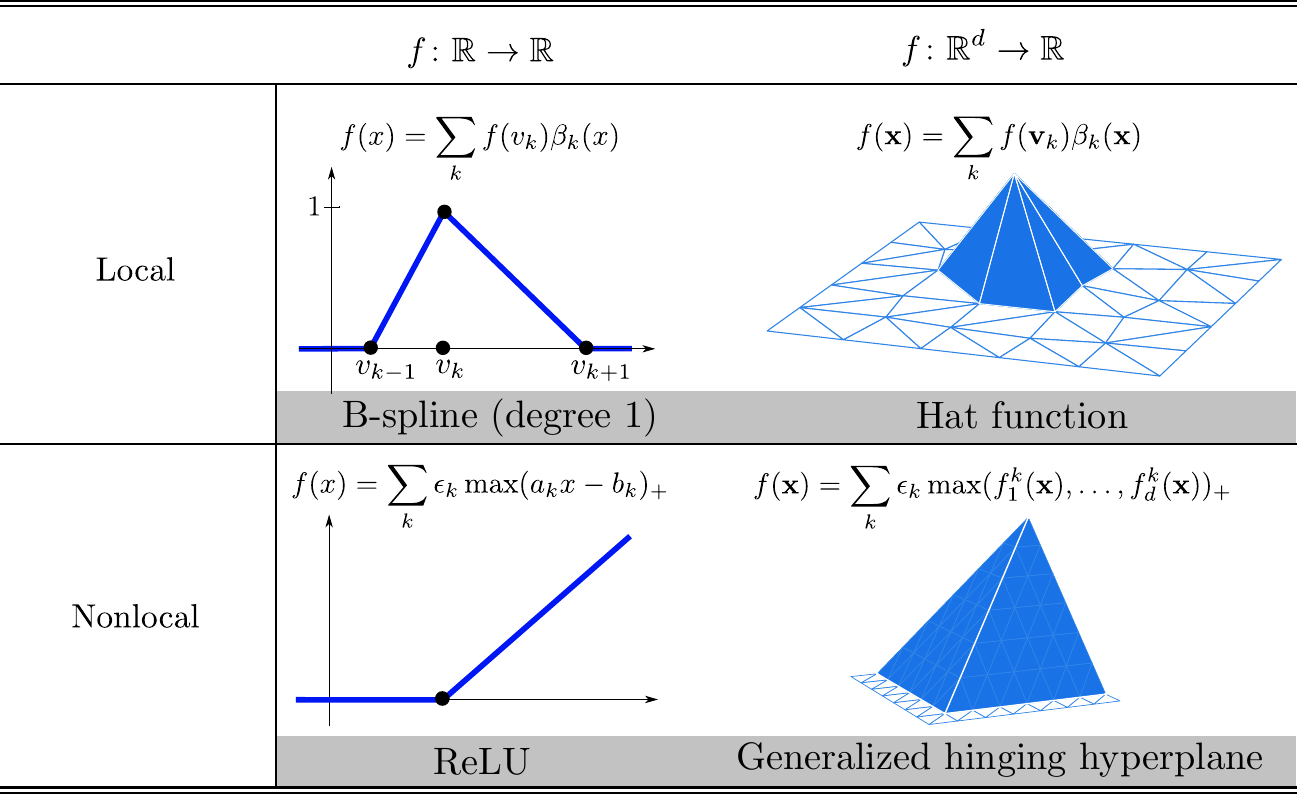}}
    \caption{Local and nonlocal building bricks of CPWL functions, from dimension 1 to any dimension. Notice that the nonlocal basis functions have several equivalent variations; only the ReLU-like one is shown in this figure.}
    \label{fig:CPWL_Representation}\medskip
  \end{minipage}
  \end{figure}
\subsubsection{Local Representation}
In dimension $d=1$, any CPWL function $f$ with knots $v_k$ can be represented by the linear expansion $f=\sum_{k} f(v_k) \beta_k$, where the $\beta_k$ are the underlying triangular B-spline basis functions \cite{DeBoor1976}. In higher dimensions, the knots are replaced by what is called a simplicial triangulation of the set $\{\M v_k\}$ of vertices, which partitions the input domain into simplices. With the appropriate triangulation, a CPWL function can be represented by the explicit linear simplicial spline expansion $f=\sum_{k} f(\M v_k) \beta_k$, where the $\beta_k$ now denote the nodal basis functions or hat functions that correspond to the triangulation \cite{He2020a} (see Section \ref{sc:simplicialCPWL} and  Figure \ref{fig:CPWL_Representation}). These functions satisfy $\beta_k(\M v_q) = \delta_{kq}$ while being affine on the simplices and compactly supported, hence the attribute {\it local}. When the vertices are regularly spaced, so that they coincide with the sites of a lattice, the hat functions can be chosen as translates of a single linear box splines \cite{de1993box,Kim2010}.
\subsubsection{Nonlocal Representations}
In dimension $d=1$, a CPWL function $f$ with control points $v_k$ can also be represented by the nonlocal representation $f(x) = a_0 + a_1(x-v_1) + \sum_{k=2}^{K-1} a_k \mathrm{ReLU}(x-v_k)$ \cite{aliprantis2006continuous}. Note that this representation has many equivalent variations such as $f(x)=c_0+c_1(x-v_1) + \sum_{k=2}^{K-1} c_k \abs{x-v_k}$. The generalization to any dimension is due to Wang and Sun \cite{Wang2005}. Their generalized hinging-hyperplanes (GHH) model can represent any CPWL function as $f(\M x) = \sum_k \epsilon_k \max (f_1^k(\M x),\ldots,f_{m_k}^k(\M x))$, where $ f_1^k,\ldots,f_{m_k}^k$ are affine functions, $\epsilon_k=\pm 1$, and $m_k\leq d+1$. This expansion has also different variations and can be recast as $f(\M x) = \sum_k \epsilon_k \max (g_1^k(\M x),\ldots,g_{m_k-1}^k(\M x))_+$, where $g_1^k,\ldots,g_{m_k-1}^k$ are affine functions, $\epsilon_k=\pm 1$, $m_k\leq d+1$ and for any $x\inR$, $(x)_+\coloneqq\max(x,0)=\mathrm{ReLU}(x)$. The basis functions of nonlocal representations are the building blocks of many feedforward neural networks. They play the role of activation functions, including ReLU, Leaky ReLU, PReLU, CReLU and maxout \cite{Lecun2015,Maas2013,He2015,Shang2016,Goodfellow2013}.
\begin{figure}[h!]
  \begin{minipage}{1.0\linewidth}
    \centering
    \centerline{\includegraphics[width=120mm]{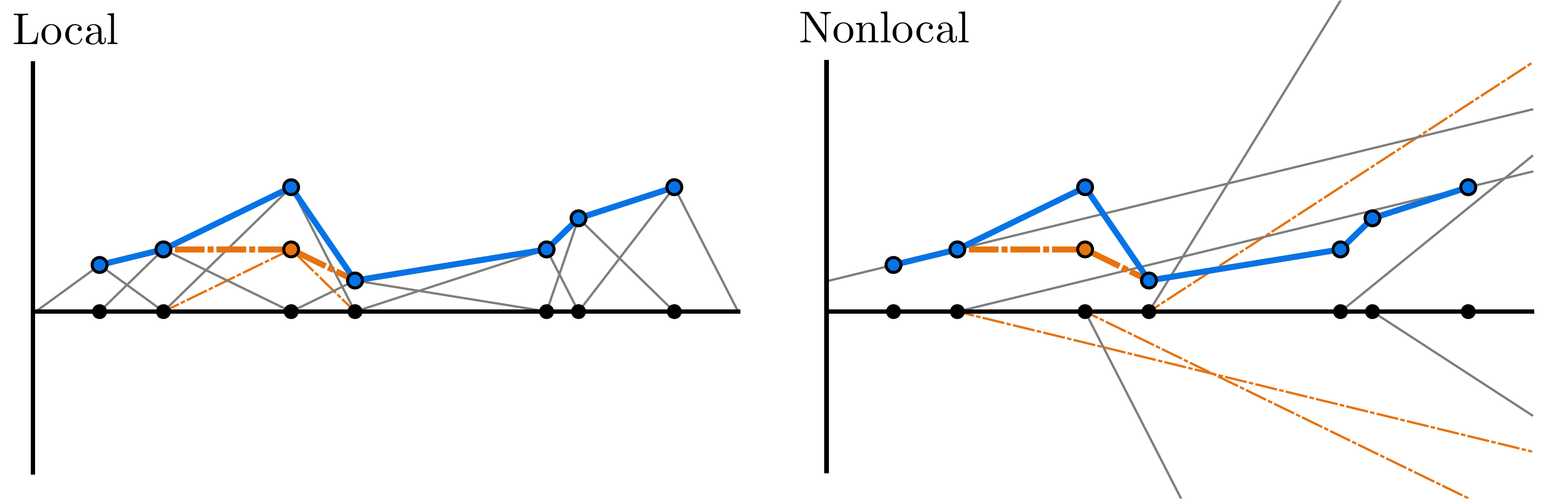}}
    \caption{Local and nonlocal linear expansion of a CPWL function (thick solid blue line) in the one-dimensional case. A point is moved (dashed thick orange line) and it results in the modification of 1 triangular basis function for the local representation and 3 ReLU-like functions for the nonlocal representation (thin dashed orange lines).}
    \label{fig:1dinterpolation}\medskip
  \end{minipage}
  \end{figure}
\subsection{On the Stability of Parametrizations}
While powerful in learning applications, nonlocal basis functions raise concerns about the stability of the model. Since the nonlocal atoms do not belong to any of the $L_p(\mathbb{R^d})$ Lebesgue spaces, the synthesis operator $T\colon \Theta \rightarrow\mathcal{H}$ is ill-conditioned. A small change in a parameter can lead to tremendous changes in the generated function. Nonlocal representations also lack stability when a CPWL function is used to interpolate data (Figure \ref{fig:1dinterpolation}). We show in Appendix \ref{ap:condition} that, in the one-dimensional case, the corresponding condition number for $K$ points is at least $\mathcal{O}(K^{3/2})$. Conversely, the local representation $f=\sum_{k} f(\M v_k) \beta_k$ offers an explicit solution with, as result, a condition number equal to 1. To further examine the stability of the local parametrization, we consider the parameter space $\ell_2(\mathbb{Z})$ of finite-energy sequences and equip the function space with the $L_2$ norm. Ideally, the collection $(\beta_k)$ of functions would form an orthonormal basis. However, this does not hold for hat functions. The next best thing one can hope for is that the generating functions form a Riesz basis, which means that $(\beta_k)$ is the image of an orthonormal basis under a bounded invertible linear operator. This guarantees that the synthesis operator $T$ is a bounded linear bijection from $\Theta$ to $\mathcal{H}$. This strong property has emerged as a standard requirement in many signal-processing theories and finite-elements methods \cite{Aldroubi1994,Aldroubi1996,Unser1997,Jia2011RieszBO,Fukuda2013} within a broad spectrum of applications.
It is known that uniform B-splines of any degree and their high-dimensional box-spline extensions (with suitable directions) generate Riesz bases \cite{dahmen1983translates,guan2005characterization}. Yet, to the best of our knowledge, the exact Riesz bounds of linear box splines are not known in high dimensions and the case of irregular triangulations has not been addressed in full generality so far.

In this paper, we propose to investigate precisely and in any dimension the stability of the local parametrization of CPWL functions with the hope to bring detailed results against which other parametrizations could be compared.

The paper is organized as follows: In Section, \ref{sc:prelim} we present the relevant mathematical concepts. We then restrict our investigation to affine functions on simplices in Section \ref{sc:AffonSimplex} and extend it to CPWL functions on any triangulation in Section \ref{sc:Irregular}. Finally, we discuss the uniform-grid setting in Section \ref{sc:linearbox-splinesStability}.

\section{Mathematical Preliminaries}
\label{sc:prelim}
\subsection{Simplicial Continuous and Piecewise-Linear Functions}
\label{sc:simplicialCPWL}
\begin{definition}
\label{df:cpwl}
A function $f \colon \R^d \rightarrow \R$ is continuous and piecewise-linear (CPWL) if it is continuous and if there exist distinct affine functions $f_1, f_2,\ldots,f_p$ and subsets
$R_1, R_2,\ldots,R_p$ of $\R^d$ such that
\begin{enumerate}[label=(\roman*)]
\item each $R_k$ is closed with nonempty interior;
\item for $k\neq q$, $R_k$ and $R_q$ have disjoint interiors;
\item the space is partitioned as $\bigcup_{k=1}^p R_k = \mathbb{R}^d$;
\item the function $f$ agrees with $f_k$ on $R_k$.
\end{enumerate}
\end{definition}
We extend this definition to compact input domains of $\R^d$ and to any function $f \colon \R^d \rightarrow \R$ whose restriction to any
compact set is CPWL (sometimes referred to as locally piecewise-affine functions \cite{adeeb2017locally}).
\begin{definition}
  A set $\mathcal{V}\subset\R^d$ is locally finite if its intersection with any compact set of $\R^d$ is finite\footnote{Note that the meaning of the term ``locally finite'' depends on the mathematical field.}.
\end{definition}
In the sequel, $\mathcal{V}$ will always denote a locally finite set of $\R^d$ indexed by $I\subset\N$ that is not contained in any $(d-1)$-dimensional affine subspace of $\R^d$. To each point $\M v_k \in \mathcal{V}$ we associate a target value $y_k \inR$. Under Definition \ref {df:cpwl}, it is not obvious to find a CPWL function $f$ satisfies $f(\M v_k) = y_k$. The local representation, on the contrary, offers a more systematic way to address this problem. This requires first to form a triangulation of the set $\mathcal{V}$.

\subsubsection{Partition of the Input Domain into Simplices}
A polyhedron is the intersection of finitely many half spaces. A polytope is a bounded polyhedron. Simplices are the polytopes that have the fewest number of faces; in growing number of dimension $d=0,\ldots,3$ they include points, segments, triangles and tetrahedrons. Formally, a $d$-simplex $s$ of $\mathbb{R}^d$ is the convex hull of $(d+1)$ affinely independent vertices
\begin{equation}
s = \conv{\M v_1,\ldots,\M v_{d+1}} = \left\{\sum_{k=1}^{d+1} \lambda_k \M v_k \colon \lambda_k\geq 0, \sum_{k=1}^{d+1} \lambda_k \M= 1\right \}.
\end{equation}
 A $k$-face of a simplex is the convex hull of $(k+1)$ of its vertices, which is a $k$-simplex embeded in $\R^d$. The volume of a simplex admits the explicit form
$
\vol{s} = \frac{1}{d!}|\mathrm{det}(
\M v_2-\M v_1,\ldots,\M v_{d+1} - \M v_1)|
$.

\begin{definition}[adapted from \cite{de2010triangulations}]
A triangulation of a locally finite set $\mathcal{V}\subset\R^d$ of points is a collection $\mathcal{S}$ of d-simplices whose vertices are points in $\mathcal{V}$ and such that
\begin{enumerate}[label=(\roman*)]
\item the union of all the simplices equals $\conv{\mathcal{V}}$ (union property);
\item any pair of simplices intersects in a (possibly
empty) common face (intersection property).
\end{enumerate}
\end{definition}
A triangulation of $\mathcal{V}$ is said to be \textit{full} if all the points of $\mathcal{V}$ are vertices of it, a property that we shall always assume to hold true in the sequel.
For a given locally finite set $\mathcal{V}\subset\R^d$ of vertices of dimension $d$, the existence of a triangulation is granted but in general not unique. In practice, most applications fall into one of the two.
\begin{itemize}
\item \textbf{Regular Sampling Locations}: The vertices coincide with the sites of a lattice, for which explicit triangulations are known ({\it e.g.}, the Kuhn triangulation \cite{kuhn1960,Allgower1978}).
\item \textbf{Irregular (Random) Sampling Locations}: A Delaunay triangulation always exists in any number of dimensions for a finite set $\mathcal{V}$, and there are efficient algorithms to compute it \cite{Watson1981, Rajan1994}.
\end{itemize}

\begin{definition}
Let $\mathcal{S}$ be a triangulation of a locally finite set $\mathcal{V}\subset\R^d$. The star $\mathrm{St}(\M v)$ of a vertex $\M v\in \mathcal{V}$ is the set of those simplices of $\mathcal{S}$ that contain $\M v$.
\end{definition}
We define the volume of the star of a vertex as $\vol{\mathrm{St}(\M v)}= \sum_{s\in\mathrm{St}(\M v)} \vol{s}$. Its cardinality is denoted by $|\mathrm{St}(\M v)|$.
\begin{figure}[h!]
\begin{minipage}{1.0\linewidth}
  \centering
  \centerline{\includegraphics[width=150mm]{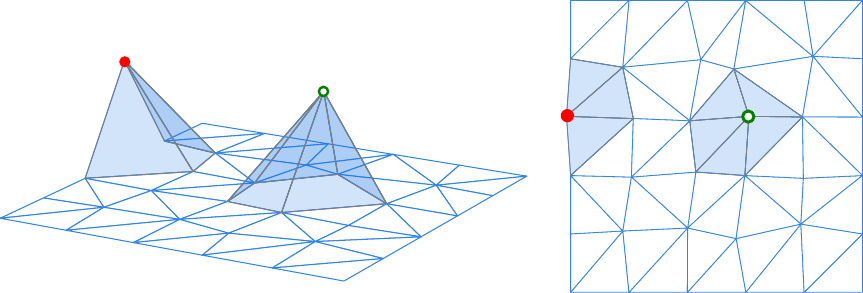}}%  \vspace{2.0cm}
  \caption{Left: 3D view of two hat functions on a finite triangulation. Note that, although the left hat function seems discontinuous, it is continuous on the triangulation. Right: star of two vertices of the triangulation (filled area).}
  \label{fig:hat_star}\medskip
\end{minipage}
\end{figure}

\subsubsection{Hat Basis Functions}
The set of CPWL functions on a triangulation $\mathcal{S}$ with vertices $\mathcal{V}$ is defined as 
\begin{equation}
\mathrm{CPWL}(\mathcal{S}) = \{f\in \mathbb{R}^{\conv{\mathcal{V}}}\colon f \hbox{ is affine on any } s\in\mathcal{S} \hbox{ and continuous on } \conv{\mathcal{V}}\}.
\end{equation}
It is said to be the space of linear simplicial splines on $\mathcal{S}$. Note that the functions in $\mathrm{CPWL}(\mathcal{S})$ are only defined over the convex hull $\conv{\mathcal{V}}$ of $\mathcal{V}$, which can range from a compact set to the whole space $\R^d$.

It is known that any polyhedron can be partitioned into simplices \cite{Edmonds1970}. As a result, any CPWL function can be viewed as a linear simplicial spline. Two affine functions that coincide on the vertices of a $d$-simplex are equal, which means that any element of $\mathrm{CPWL}(\mathcal{S})$ is uniquely determined by the values it assumes at the vertices $\mathcal{V}$ of $\mathcal{S}$. This leads to the local linear expansion
\begin{equation}
\forall f \in \mathrm{CPWL}(\mathcal{S})\colon f= \sum_{\M v\in \mathcal{V}}f(\M v)\beta^{\mathcal{S}}_{\M v},
\end{equation}
where the hat functions $\beta^{\mathcal{S}}_{\M v}\in \mathbb{R}^{\conv{\mathcal{V}}}$ (see Figure \ref{fig:hat_star}) are defined on every simplex $s\in\mathcal{S}$ by
\begin{equation}
\label{eq:linear_exp}
{\beta^{\mathcal{S}}_{\M v}}_{|s} = \begin{cases}
\lambda_{\M v}^s, & s\in\mathrm{St}(\M v)\\
0, & \text{ otherwise},
\end{cases}
\end{equation}
where $\lambda_{\M v}^s$ is the unique affine function that is vanishes at all vertices of $s$ but takes value $1$ at vertex $\M v$. In other words, $\lambda_{\M v}^s$ outputs the barycentric coordinate of simplex $s$ attached to vertex $\M v$ for a given $\M x\in s$. Note that depending on the set of vertices, the hat basis functions might not be defined over the whole $\R^d$ and, in the sequel, for any $f\in \mathbb{R}^{\conv{\mathcal{V}}}$ we use the notation $\|f\|_{L_p} = (\int_{\M x \in \conv{\mathcal{V}}} |f(\M x)|^p\dint \M x)^{1/p}$.
The hat basis functions have many desirable properties, such as,
\begin{itemize}
  \item for $\M u,\M v \in \mathcal{V}, \beta^{\mathcal{S}}_{\M v}(\M u) = \begin{cases}
    1,& \M v = \M u\\
    0,& \text{otherwise};
  \end{cases}$
\item continuity;
\item compact support $\mathrm{supp}(\beta^{\mathcal{S}}_{\M v}) = \mathrm{St}(\M v)$;
\item minimal support among all nonzero functions of $\mathrm{CPWL}(\mathcal{S})$;
\item ability to reproduce polynomials of degree up to $1$ on $\conv{\mathcal{V}}$, so that
\begin{equation}
\forall (\M a,b) \inR^d\times\R, \;\forall \M x\in\conv{\mathcal{V}}\colon \M a^T \M x + b = \sum_{v\in \mathcal{V}}(a^T \M v + b)\beta^{\mathcal{S}}_{\M v}(\M x),
\end{equation}
which includes the partition-of-unity condition $\sum_{v\in \mathcal{V}}\beta^{\mathcal{S}}_{\M v}=1$.
\end{itemize}

When $\mathrm{St}(\M v)$ is convex, the hat function simply reads \cite{He2020}
\begin{equation}
\beta^{\mathcal{S}}_{\M v} =
\left(\min_{s\in \mathrm{St}(\M v)} \lambda_{\M v}^s\right)_+.
\end{equation}

\subsection{Riesz Bases}
For a set $I\subset\N$, we denote by $\ell_2(I)$ the set of complex-valued sequences indexed by $I$ with finite energy.
\begin{definition}
  Let $\mathcal{H}$ be a separable Hilbert space over $\mathbb{C}$ and $I\subset \N$. A collection of functions $\{\varphi_k\}_{k\in I}$ in $\mathcal{H}$ is a Riesz basis if
  \begin{enumerate}[label=(\roman*)]
    \item $\overline{\mathrm{Span}(\{\varphi_k\}_{k\in I})} = \mathcal{H}$ (completeness),
    \item there exist $0<A\leq B<+\infty$ such that, for any $c\in \ell_2(I)$, 
    \begin{equation}
      \label{eq:RBdf}
      A\| c\|_{\ell_2}\leq \|\sum_{k\in I} c_k\varphi_k\|_{L_2}\leq B \| c\|_{\ell_2} \text{ (Riesz sequence property)},
    \end{equation}
    \end{enumerate}
    where $\| c\|_{\ell_2} =\left(\sum_{k\in I}\abs{c_k}^2\right)^{1/2}$.
\end{definition}
The tightest constants $A$ and $B$ that satisfy \eqref{eq:RBdf} are called the {\it Riesz bounds}. We call the ratio $B/A$ the {\it Riesz condition number}. It is indeed the $\ell_2\rightarrow L_2$ condition number of the synthesis operator $T\colon c \mapsto \sum_{k\in I}c_k\varphi_k$. The Riesz-basis property guarantees that $T$ is a bounded linear bijection, which means that there is a unique and stable link between the parameters and the functions being generated. Note that the condition number is $1$ if and only if the collection of functions $(\varphi_k)_{k\in I}$ forms an orthonormal basis (up to a scaling factor).

When the collection of functions is formed by the multi-index shifts of a single generating function ($\{\varphi_{\M k}\}=\{\varphi(\cdot-\M k)\colon \M k\inZ^d\})$, the Riesz-sequence property is well characterized via the discrete-time Fourier transform $\widehat g$ of the sampled autocorrelation of $\varphi$, as given by
\begin{equation}
  \widehat g\colon \V \omega \mapsto \sum_{\M k\inZ^d}\dotprod{\varphi}{\varphi(\cdot-\M k)}\ee^{-\ii\M k^T \V \omega}.
\end{equation}
In this uniform scenario, the Fourier equivalent of the
Riesz sequence condition is \cite{Aldroubi1996,vandevilleHexSplinesNovelSpline2004}
\begin{equation}
  \label{eq:RBFourier}
  0 < A^2 = \essinf_{\V \omega\in [0,2\pi]^d}\widehat g(\V \omega)\leq B^2 = \esssup_{\V \omega\in [0,2\pi]^d}\widehat g(\V \omega)<+\infty.
\end{equation}
\section{Affine Functions on Simplices}
\label{sc:AffonSimplex}
Considerations on affine functions on simplices will help us lay the foundations of the analysis of the stability of the local parametrization on simplicial partitions (Sections \ref{sc:Irregular} and \ref{sc:linearbox-splinesStability}).
\begin{proposition}
  \label{pr:extremeL2simplex}
  Let $f:\mathbb{R}^d\rightarrow \mathbb{C}$ be an affine function and $s = \conv{\M v_1,\ldots,\M v_{d+1}}$ a $d$-simplex. Then
  \begin{align}
  \frac{\vol{s}}{(d+2)(d+1)} \sum_{k=1}^{d+1} |f(\M v_k)|^2 \leq \int_{\M x \in s} |f(\M x)|^2 \dint \M x \leq  \frac{\vol{s}}{(d+1)} \sum_{k=1}^{d+1} |f(\M v_k)|^2,
  \end{align}
  \end{proposition}
Prior to proving Proposition \ref{pr:extremeL2simplex}, we provide a series of useful results regarding the computation of some integrals of affine functions over simplices.

For a linear function $f\colon \mathbb{R}^d\rightarrow \mathbb{R}$, an integer $p\inN$, and a $d$-simplex $s=\conv{\M v_1,\ldots,\M v_{d+1}}$, it is known that \cite{Lasserre2001,Baldoni2010} 
\begin{equation}
\label{eq:integral-linpower-big-sum}
\int_s f(\M x)^p \dint \M x=  \vol{s} {p+d \choose d}^{-1}
\sum_{\M k \in \N^{d+1}, |\M k| = p}f(\M v_1)
^{k_1}\cdots f(\M v_{d+1})^{k_{d+1}},
\end{equation}
where we use the notation $\M k=(k_1,\ldots,k_{d+1})$ and $\abs{\M k}:= k_1+\cdots+k_{d+1}$.
We can extend this to affine functions.
\begin{lemma}
\label{lm:Lpsimplex}
Let $f:\mathbb{R}^d\rightarrow \mathbb{R}$ be an affine function and $s= \conv{\M v_1,\ldots,\M v_{d+1}}$ a $d$-simplex. For any $p\inN$, we have that
\begin{equation}
\int_{s} f(\M x)^p \dint \M x= \vol{s} {p+d \choose d}^{-1} 
\sum_{\M k \in \N^{d+1}, |\M k| = p}f(\M v_1)
^{k_1}\cdots f(\M v_{d+1})^{k_{d+1}}.
\end{equation}
\end{lemma}
\begin{proof}
If the affine function $f$ is not constant, then it can be written as $f(\M x)=\M a^T (\M x - \M x_0)$. Equation \eqref{eq:integral-linpower-big-sum} can be applied after a change of variable, as in
\begin{align}
\int_{s} f(\M x)^p \dint \M x &= \int_{s} (\M a^T (\M x - \M x_0))^p \dint \M x = \int_{s-\M x_0} (\M a^T \M y)^p \dint \M y \\&= \vol{s-\M x_0} {p+d \choose d}^{-1}
\sum_{\M k \in \N^{d+1}, |\M k| = p}(\M a^T (\M v_{1} - \M x_0))
^{k_1}\cdots (\M a^T (\M v_{d+1} - \M x_0))^{k_{d+1}}\\
&=\vol{s} {p+d \choose d}^{-1} 
\sum_{\M k \in \N^{d+1}, |\M k| = p}f(\M v_1)
^{k_1}\cdots f(\M v_{d+1})^{k_{d+1}}.
\end{align}
where $s-\M x_0 \coloneqq \{\M x- \M x_0 \colon \M x \in s\}$.
If now the affine function $f$ is constant with $f(\M x) = b$, then $\int_{s} f(\M x)^p \dint \M x =  \vol{s}b^p$. We also have that 
\begin{equation}
\vol{s} {p+d \choose d}^{-1} \sum_{\M k \in \N^{d+1}, |\M k| = p}f(\M v_1)
^{k_1}\cdots f(\M v_{d+1})^{k_{d+1}} = \vol{s} {p+d \choose d}^{-1}\sum_{\M k \in \N^{d+1}, |\M k| = p} b^p= \vol{s} b^p,
\end{equation}
where we have used that $\sum_{\M k \in \N^{d+1}, |\M k| = p} 1= {p+d \choose d}$. This number is known in combinatorics as the combinations with replacement \cite{Heumann2016}. 
\end{proof}
We can now deduce an important property of the hat functions.
\begin{proposition}
\label{pr:Lpnorm}
The $L_p$ norm of the hat function $\beta^{\mathcal{S}}_{\M v}$ only depends on the dimension $d$ and the volume of its support. It reads
\begin{equation}
\|\beta^{\mathcal{S}}_{\M v}\|_{L_p} = \left({p+d \choose d}^{-1}\vol{\mathrm{St}(\M v)}\right )^{1/p}.
\end{equation}
\end{proposition}
\begin{proof}
We split the integral over the simplices of the support of $\beta^{\mathcal{S}}_{\M v}$ and apply Lemma \ref{lm:Lpsimplex}, which leads
\begin{align}
\|\beta^{\mathcal{S}}_{\M v}\|_{L_p} ^p &=  \int_{\conv{\mathcal{V}}} |\beta^{\mathcal{S}}_{\M v}(\M x)|^p\dint \M x = \int_{\conv{\mathcal{V}}} \beta_{\M v}(\M x)^p\dint \M x \nonumber \\
&=\sum_{s\in \mathcal{S}} \int_{s}\beta^{\mathcal{S}}_{\M v}(\M x)^p\dint \M x = \sum_{s\in \mathrm{St}(\M v)} \int_{s}\beta^{\mathcal{S}}_{\M v}(\M x)^p\dint \M x \nonumber\\
&=\sum_{s\in \mathrm{St}(\M v)} {p+d \choose d}^{-1}\vol{s}= {p+d \choose d}^{-1}\vol{\mathrm{St}(\M v)}.
\end{align}

\end{proof}
In the sequel, we shall make use of Proposition \ref{pr:Lpnorm} through the two following relations:
\begin{itemize}
\item the $L_2$ norm \begin{equation}
\label{fm:L2norm}
\|\beta^{\mathcal{S}}_{\M v}\|_{L_2}^2 = \frac{2\vol{\mathrm{St}(\M v)}}{(d+1)(d+2)};
\end{equation}
\item the inner-product relation
\begin{equation}
  \label{eq:L1norm}
\dotprod{\beta^{\mathcal{S}}_{\M v}}{\sum_{\M u \in \mathcal{V}}\beta^{\mathcal{S}}_{\M u}}_{\conv{\mathcal{V}}} = \dotprod{\beta^{\mathcal{S}}_{\M v}}{1}_{\conv{\mathcal{V}}} = \|\beta^{\mathcal{S}}_{\M v}\|_{L_1} = \frac{\vol{\mathrm{St}(\M v)}}{(d+1)},
\end{equation}
where $\dotprod{f}{h}_{\conv{\mathcal{V}}} = \int_{\M x\in {\conv{\mathcal{V}}}}\overline{f(\M x)}h(\M x)\dint \M x$ and where the first equality results from the partition of unity of the hat functions.
\end{itemize}

Interestingly, when $p=2$, the integral in Lemma \ref{lm:Lpsimplex} is a quadratic form of the value of the function on the vertices and admits the matrix form shown in Lemma \ref{lm:L2simplex}.
\begin{lemma}
\label{lm:L2simplex}
Let $f:\mathbb{R}^d\rightarrow \mathbb{R}$ be an affine function, $s= \conv{\M v_1,\ldots,\M v_{d+1}}$ a $d$-simplex, and $\M f_s =( f(\M v_1),\cdots , f(\M v_{d+1}))\inR^{d+1}$. Then,
\begin{equation}
\int_{s} f(\M x)^2 \dint \M x = \frac{\vol{s}}{(d+1)(d+2)}  \M f_s^T \M P_{d+1} \M f_s,
\end{equation}
where
\begin{equation}
\M P_{d+1} =  \M 1_{d+1}  + \M I_{d+1}  = \begin{bmatrix}
2 &1 &\ldots&1\\
1 &\ddots & \ddots & \vdots\\
\vdots& \ddots & \ddots & 1\\
1 & \cdots & 1 & 2
\end{bmatrix}\inR^{(d+1)\times(d+1)}.
\end{equation}
\end{lemma}
\begin{proof}
Following Lemma \ref{lm:Lpsimplex}, on one hand we have that
\begin{align}
\int_{s} f(\M x)^2 \dint \M x &= \vol{s} {2+d \choose d}^{-1} 
\sum_{\M k \in \N^{d+1}, |\M k| = 2}f(\M v_1)
^{k_1}\cdots f(\M v_{d+1})^{k_{d+1}}\nonumber\\
&=\frac{2 \vol{s}}{(d+1)(d+2)}1/2 
\left(\sum_{p,q=1}^{d+1} f(\M v_p)f(\M v_q) + \sum_{p=1}^{d+1} f(\M v_p)^2 \right) \nonumber \\
&=\frac{\vol{s}}{(d+1)(d+2)}
\left(\left(\sum_{p=1}^{d+1} f(\M v_q)\right)^2 + \sum_{p=1}^{d+1} f(\M v_p)^2 \right).
\end{align}
On the other hand, we have that
\begin{align}
\M f_s^T \M P_{d+1} \M f_s &= \sum_{p=1}^{d+1} f(\M v_p)\left(\sum_{q=1}^{d+1} f(\M v_q) + f(\M v_p)\right)\\
&= \left(\sum_{p=1}^{d+1} f(\M v_p) \right)^2 + \sum_{p=1}^{d+1} f(\M v_p)^2.
\end{align}
\end{proof}
Lemma \ref{lm:L2simplex} can be extended to pairs of complex-valued functions.
\begin{lemma}
\label{lm:innerprodcutsimplexcomplex}
Let $f,g:\mathbb{R}^d\rightarrow \mathbb{C}$ be affine functions, $s= \conv{\M v_1,\ldots,\M v_{d+1}}$ a $d$-simplex, $\M f_s =( f(\M v_1),\cdots , f(\M v_{d+1}))\inR^{d+1}$, and $\M g_s =( g(\M v_1),\cdots , g(\M v_{d+1}))\inR^{d+1}$. It holds that
\begin{equation}
\int_{s} \overline{f(\M x)}g(\M x) \dint \M x = \frac{\vol{s}}{(d+1)(d+2)}  \M f_s^H \M P_{d+1} \M g_s,
\end{equation}
where
\begin{equation}
\M P_{d+1} =  \M 1_{d+1}  + \M I_{d+1} \inR^{(d+1)\times(d+1)}.
\end{equation}
\end{lemma}
To prove Lemma \ref{lm:innerprodcutsimplexcomplex}, we first consider real-valued functions $f$ and $g$ and apply Lemma \ref{lm:L2simplex} to the left-hand side of the equality $2fg=\left((f+g)^2-f^2-g^2\right)$. The announced result is then reached because $\M P_{d+1}$ is a symmetric matrix. The generalization to complex-valued functions is directly obtained via the decomposition of $f$ and $g$ into their real and imaginary parts.

\begin{proof}[Proof of Proposition \ref{pr:extremeL2simplex}]
The matrix $\M P_{d+1} $ defined in Lemma \ref{lm:L2simplex} is a symmetric circulant matrix generated by the vector $(2,1,\ldots,1)$. Its eigenvalues are known to be \cite{Kra2012}
\begin{equation}
\lambda_m = 2 + \sum_{n=1}^{d}  \zeta_{d+1}^{mn} \quad \text{ with } m=1,\ldots,d+1 \text{ and where } \zeta_{d+1}= \ee^{\ii \frac{2\pi}{d+1}}.
\end{equation}
These expressions are further simplified to
\begin{align}
\lambda_m = 1 + \sum_{n=1}^{d+1}  \zeta_{d+1}^{mn} = \begin{cases}
d+2, & m = d+1\\
1, & \text{otherwise},
\end{cases}
\end{align}
which shows that $\min_{m\in\{1,\ldots,d+1\}}(\lambda_m) = 1$ and $\max_{m\in\{1,\ldots,d+1\}}(\lambda_m) = (d+2)$. As a result, for any $\M c \in\mathbb{C}^{d+1}$
\begin{equation}
\| \M c\|_{2}^2 \leq \M c^H\M P_{d+1}\M c \leq (d+2) \| \M c\|_{2}^2.
\end{equation}
We now conclude by applying Lemma \ref{lm:innerprodcutsimplexcomplex}.
\end{proof}
Unfortunately, the condition number $\sqrt{d+2}$ given by the inequalities in Proposition \ref{pr:extremeL2simplex} depends on the dimension $d$. However, the eigenvalues of $\M P_{d+1} $ are all $1$ except for one that is $(d+2)$. This means that, in general, $\frac{\sqrt{\M c^H\M P_{d+1}\M c}}{\| \M c\|}$ is very heavily distributed toward $1$ and only rarely approaches the upper bound $\sqrt{d+2}$, especially in high dimensions, as quantified in Lemma \ref{lm:conditionnumberdistribution}. One should recall that, even though the dimension worsens the condition number, the effect is stochastically negligible.
\begin{proposition}
  \label{lm:conditionnumberdistribution}
  Let $\M P_{d+1} =  \M 1_{d+1}  + \M I_{d+1} \inR^{(d+1)\times(d+1)}$, $\V C$ a random unit vector of $\C^{d+1}$ whose distribution is uniform on the unit sphere. Then, the first two moments of the random variable $\frac{\sqrt{\V C^H\M P_{d+1}\V C}}{\|\V C\|_2}$ can be bounded by quantities that do not depend on the dimension $d$, as expressed by
  \begin{equation}
    \mathbb{E}\left(\frac{\sqrt{\V C^H\M P_{d+1}\V C}}{\|\V C\|_2}\right) \leq \sqrt{2} \;\text{  and  }\;\mathbb{E}\left(\frac{\V C^H\M P_{d+1}\V C}{\|\V C\|_2^2}\right) = 2.
  \end{equation}
\end{proposition}
\begin{proof}
  Knowing that $\M P_{d+1}$ is a real and symmetric matrix, there exists an orthonormal basis $(\M u_k)_{k=1}^{d+1}$ made of its eigenvectors. They can be chosen so that the associated eigenvalues are $1,\ldots,1,(d+2)$ (see proof of Proposition \ref{pr:extremeL2simplex}). The random vector $\V C$ can be decomposed as $\V C=\sum_{k=1}^{d+1}A_k \M u_k$, where the coordinates $A_k$ are random scalar variables that follow the same distribution and satisfy $\sum_{k=1}^{d+1}|A_k|^2=1$. We directly deduce that $\mathbb{E}(|A_k|^2)=\frac{1}{d+1}$. In addition, $\V C^H\M P_{d+1}\V C=\sum_{k=1}^{d}|A_k|^2 + (d+2)|A_{d+1}|^2=1+(d+1)|A_{d+1}|^2$, which yields that $\mathbb{E}\left(\frac{\V C^H\M P_{d+1}\V C}{\|\V C\|_2^2}\right) = 2$ since $\V C^H \V C = 1$. The conclusion follows from the well known inequality $\mathbb{E}(X)^2<\mathbb{E}(X^2)$ true for any random variable $X$.
\end{proof}
\section{Stability of the Local Parametrization on Irregular Triangulations}
\label{sc:Irregular}
\subsection{Triangulations with Any Number of Vertices}
\label{sc:irregularStability}
Triangulations in high dimensions have complex combinatorial structures that can induce a wide range of behaviors with the usual descriptors, for instance shape of the simplices, degree of the vertices, number of simplices shared by $2,\ldots,d$ vertices. Fortunately, the necessary and sufficient condition that a triangulation must satisfy for the local hat functions to form a Riesz basis only relies on a simple quantity: the volume of the star of the vertices.
\begin{theorem}
\label{th:RBIrregular}
Let $\mathcal{S}$ be a triangulation of a locally finite set $\mathcal{V}=\{\M v_k\}_{k\in I}$ of vertices in $\R^d$ and let $(\beta^{\mathcal{S}}_{\M v})_{\M v\in\mathcal{V}}$ be the corresponding hat functions. Then, the following statements are equivalent:
\begin{enumerate}[label=(\roman*)]
\item  the collection of functions $(\beta^{\mathcal{S}}_{\M v})_{\M v\in\mathcal{V}}$ forms a Riesz basis of the space $\mathrm{CPWL}(\mathcal{S})\cap L_2(\R^d)$;
\item $\begin{cases}V^{\mathrm{St}}_{\inf} = \inf_{\M v \in \mathcal{V}}\vol{\mathrm{St}(\M v)} > 0 \\
V^{\mathrm{St}}_{\sup} = \sup_{\M v \in \mathcal{V}}\vol{\mathrm{St}(\M v)} < + \infty;
\end{cases}$
\item $\begin{cases}\inf_{\M v \in \mathcal{V}}\|\beta^{\mathcal{S}}_{\M v}\|_{L_2} > 0 \\
\sup_{\M v \in \mathcal{V}}\|\beta^{\mathcal{S}}_{\M v}\|_{L_2} < + \infty.
\end{cases}$
\end{enumerate}
When these statements hold, for any $c \in \ell_2(I)$ we have that
\begin{equation}
\sqrt{\frac{V^{\mathrm{St}}_{\inf}}{(d+1)(d+2)}}\|c\|_{\ell_2} \leq \|\sum_{v\in\mathcal{V}}c_{\M v} \beta^{\mathcal{S}}_{\M v}\|_{L_2}\leq \sqrt{\frac{V^{\mathrm{St}}_{\sup}}{(d+1)}}\| c\|_{\ell_2}
\end{equation}
or, equivalently, that
\begin{equation}
\frac{1}{\sqrt{2}}\inf_{\M v \in \mathcal{V}}\|\beta^{\mathcal{S}}_{\M v}\|_{L_2}\| c\|_{\ell_2} \leq \|\sum_{v\in\mathcal{V}}c_{\M v} \beta^{\mathcal{S}}_{\M v}\|_{L_2}\leq \sqrt{\frac{d+2}{2}} \sup_{\M v \in \mathcal{V}}\|\beta^{\mathcal{S}}_{\M v}\|_{L_2} \| c\|_{\ell_2}.
\end{equation}
The Riesz condition number $r$ satisfies
\begin{equation}
r \leq \sqrt{d+2}\sqrt{\frac{V^{\mathrm{St}}_{\sup}}{V^{\mathrm{St}}_{\inf}}} = \sqrt{d+2}\times \frac{ \sup_{\M v \in \mathcal{V}}\|\beta^{\mathcal{S}}_{\M v}\|_{L_2}}{ \inf_{\M v \in \mathcal{V}}\|\beta^{\mathcal{S}}_{\M v}\|_{L_2}}.
\end{equation}
\end{theorem}
\begin{proof}
The equivalence $(ii) \Leftrightarrow (iii)$ is a direct consequence of \eqref{fm:L2norm}.

We now show that $(i)\Rightarrow (iii)$. We consider sequences that are zero everywhere, but at one location, and use the Riesz-sequence property to deduce that, for any $\M v\in \mathcal{V}$,
\begin{equation}
  A\leq \|\beta^{\mathcal{S}}_{\M v}\|_{L_2}\leq B,
\end{equation}
where $0\le A\leq B\le + \infty$ are the Riesz bounds.

We now prove $(ii)\Rightarrow(i)$. Let $c \in \ell_2(I)$, $f=\sum_{\M v\in \mathcal{V}}c_{\M v} \beta^{\mathcal{S}}_{\M v}$. We have that
\begin{align}
\|f\|_{L_2}^2 &= \int_{\conv{\mathcal{V}}} |f(\M x)|^2 \dint \M x=\sum_{s\in \mathcal{S}}\int_{s} |f(\M x)|^2 \dint \M x \nonumber\\
&\leq \frac{1}{(d+1)} \sum_{s\in \mathcal{S}} \sum_{\M v\in s\cap\mathcal{V}} \vol{s} |f(\M v)|^2 \nonumber\\
&= \frac{1}{(d+1)}  \sum_{\M v \in \mathcal{V}}\sum_{s \in \mathrm{St}(\M v)}\vol{s} |f(\M v)|^2 \nonumber\\
&= \frac{1}{(d+1)}  \sum_{\M v \in \mathcal{V}}|f(\M v)|^2\sum_{\substack{S \in \mathrm{St}(\M v)}}\vol{s} \nonumber\\
&= \frac{1}{(d+1)}  \sum_{\M v \in \mathcal{V}}\vol{\mathrm{St}(\M v)} |f(\M v)|^2 \nonumber \\
&\leq \frac{\sup_{\M v \in \mathcal{V}}(\vol{\mathrm{St}(\M v)})}{(d+1)}\|c\|_{\ell_2}^2,
\end{align}
where we have applied Proposition \ref{pr:extremeL2simplex} and interchanged the order of the double summation with positive arguments (special case of Tonelli's theorem). Similarly,
\begin{align}
\|f\|_{L_2}^2 &= \int_{\conv{\mathcal{V}}} |f(\M x)|^2 \dint \M x=\sum_{s\in \mathcal{S}}\int_{s} |f(\M x)|^2 \dint \M x \nonumber\\
&\geq \frac{1}{(d+1)(d+2)} \sum_{s\in \mathcal{S}} \sum_{\M v\in s\cap\mathcal{V}} \vol{s} |f(\M v)|^2\nonumber\\
&= \frac{1}{(d+1)(d+2)}  \sum_{\M v \in \mathcal{V}}\sum_{s \in \mathrm{St}(\M v)}\vol{s} |f(\M v)|^2\nonumber\\
&= \frac{1}{(d+1)(d+2)}  \sum_{\M v \in \mathcal{V}}\vol{\mathrm{St}(\M v)} |f(\M v)|^2\nonumber\\
&\geq \frac{\inf_{\M v \in \mathcal{V}}(\vol{\mathrm{St}(\M v)})}{(d+1)(d+2)} \|c\|_{\ell_2}^2.
\end{align}
\end{proof}
Theorem \ref{th:RBIrregular} provides a quantitative way to compare the stability of triangulations. In particular, a triangulation is good when the ratio $\frac{V^{\mathrm{St}}_{\sup}}{V^{\mathrm{St}}_{\inf}}$ is close to $1$, which is an indicator of how uniform the triangulation is.

From Theorem \ref{th:RBIrregular}, we deduce a stronger condition that is sufficient for the Riesz property to hold and that gives further insight into the problem. Since $\vol{\mathrm{St}(\M v)}= \sum_{s\in\mathrm{St}(\M v)} \vol{s}$, we have that
\begin{equation}
  \inf_{s\in \mathcal{S}}\vol{s}\times \inf_{\M v\in \mathcal{V}}|\mathrm{St}(\M v)| \leq V^{\mathrm{St}}_{\inf}\leq V^{\mathrm{St}}_{\sup}\leq\sup_{s\in \mathcal{S}}\vol{s}\times \sup_{\M v\in \mathcal{V}}|\mathrm{St}(\M v)|.
\end{equation}
This means that the hat functions form a Riesz basis whenever the degree of the vertices is upper-bounded and the volume of the simplices is upper- and lower-bounded. This condition, however, is not necessary. 

Theorem \ref{th:RBIrregular} does not contain direct information on the quality of the bounds. Yet, we shall prove in Section \ref{sc:linearbox-splinesStability} that, when the hat functions are shifts of a linear box spline, the given bounds are optimal.
\subsection{Triangulations with Finitely Many Vertices}
For a triangulation $\mathcal{S}$ of a finite set $\mathcal{V}$ of vertices in $d$ dimensions, the hat functions always form a Riesz basis. Indeed, since we assumed that $\mathcal{V}$ cannot be contained in any $(d-1)$-dimensional affine subspace of $\R^d$, all simplices of $\mathcal{S}$ are not degenerated and the conditions of Theorem \ref{th:RBIrregular} are fulfilled. The exact bounds can be computed in accordance with Theorem \ref{th:explicitboundsfinte}.
\begin{theorem}
  \label{th:explicitboundsfinte}
  Let $\mathcal{S}$ be a triangulation of a finite set $\mathcal{V}$ of dimension $d$ of vertices. The corresponding hat functions form a Riesz basis of $\mathrm{CPWL}(\mathcal{S})$ with bounds
  \begin{equation}
    A = \sqrt{\frac{\lambda_{\min}(\M M)}{(d+1)(d+2)}}  \text{   and   } B = \sqrt{\frac{\lambda_{\max}(\M M)}{(d+1)(d+2)}} ,
  \end{equation}
  where the matrix $\M M\inR^{|\mathcal{V}|\times |\mathcal{V}|}$ is symmetric and defined as
  \begin{equation}
    {[\M M]}_{pq}=\begin{cases}
      2\vol{\mathrm{St}(\M v_p)}, & p=q\\
      \vol{\mathrm{St}(\M v_p)\cap \mathrm{St}(\M v_q)}, & \text{otherwise},
    \end{cases}
  \end{equation}
  and where $\lambda_{\min}$ and $\lambda_{\max}$ are the smallest and largest eigenvalues of $\M M$.
\end{theorem}
\begin{proof}
  Let $f=\sum_{\M v\in \mathcal{V}}c_{\M v}\beta^{\mathcal{S}}_{\M v}$. The coefficients $c_{\M v}$ are ordered to form a vector $\M c\inR^{|\mathcal{V}|}$. To each $s\in\mathcal{S}$ we associate a matrix $\M L_s\inR^{(d+1)\times |\mathcal{V}|}$ such that $\M L_s\M c \inR^{d+1}$ contains the coefficients associated with the vertices of $s$. By invoking Lemma \ref{lm:L2simplex}, we obtain
  \begin{align}
    \|f\|_{L_2}^2 &= \int_{\conv{\mathcal{V}}} |f(\M x)|^2 \dint \M x=\sum_{s\in \mathcal{S}}\int_{s} |f(\M x)|^2 \dint \M x \nonumber\\
    &=\frac{1}{(d+1)(d+2)}\sum_{s\in \mathcal{S}} \vol{s}(\M L_s \M c)^H \M P_{d+1} (\M L_s \M c)\nonumber\\
    &= \frac{1}{(d+1)(d+2)} \M c^H \left(\sum_{s\in \mathcal{S}} \vol{s}\M L_s^T \M P_{d+1}\M L_s\right) \M c\nonumber\\
    &= \frac{1}{(d+1)(d+2)} \M c^H \M M \M c,
  \end{align}
where $\M M = \sum_{s\in \mathcal{S}} \vol{s}\M L_s^T \M P_{d+1}\M L_s$. Let $1\leq p,q\leq |\mathcal{V}|$. For $p\neq q$, each entry $(p,q)$ of $\vol{s}\M L_s^T \M P_{d+1}\M L_s$ is given by
\begin{itemize}
  \item $\vol{s}$, if and only if $\M v_p$ and $\M v_q$ are in $s$ or, equivalently, $s\in\mathrm{St}(\M v_p)\cap \mathrm{St}(\M v_q)$;
  \item $0$, otherwise.
\end{itemize}
Now, for $p=q$, each term $\vol{s}\M L_s^T \M P_{d+1}\M L_s$ of the sum has the entry $(p,p)$ be
\begin{itemize}
  \item 2$\vol{s}$, if and only if $\M v_p$ is in $s$ or, equivalently $s\in\mathrm{St}(\M v_p)$;
  \item $0$, otherwise.
\end{itemize}
This shows that $\M M$ is precisely the matrix given in Theorem \ref{th:explicitboundsfinte}.
\end{proof}
For a finite set $\mathcal{V}$ of vertices, one may wonder which triangulation yields the most stable CPWL model in the $L_2$ sense ({\it i.e.}, the smallest $\ell_2\rightarrow L_2$ condition number). The Delaunay triangulation is known to be optimal in several ways ({\it e.g.}, in the plane it maximizes the minimum angle of all the angles of the triangles in the triangulation), but it does not necessarily give the smallest condition number. When the Delaunay triangulation is not unique, the choice can be guided by the related condition number, as detailed in Figure \ref{fig:cn_Delaunay}.

\begin{figure}[h!]
  \begin{minipage}{1.0\linewidth}
    \centering
    \centerline{\includegraphics[width=180mm]{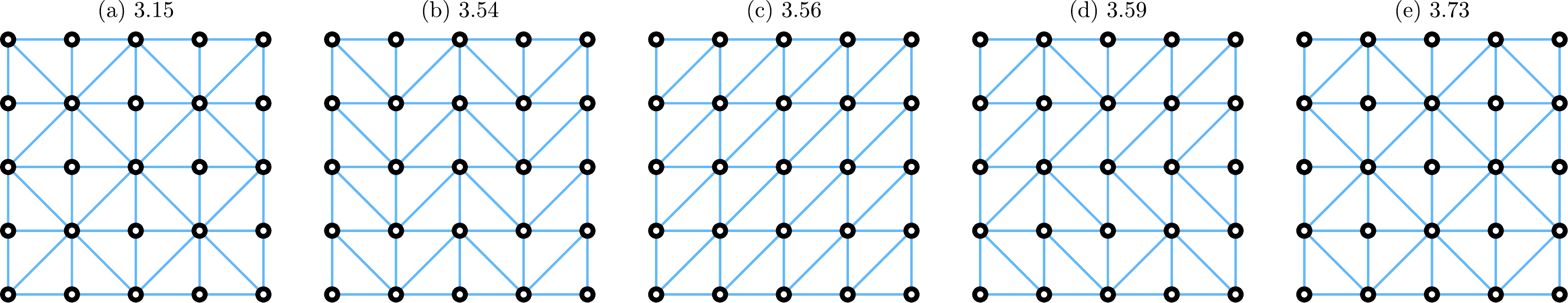}}%  \vspace{2.0cm}
    \caption{Condition number of the local parametrization of CPWL functions on various Delaunay triangulations for the same set of vertices. The results stem from Theorem \ref{th:explicitboundsfinte} and were obtained numerically. In a) and e), the pattern is similar but the condition numbers differ significantly. This comes mainly from the behavior on the border since the smallest star has two simplices for a) while it has a single one for e).}
    \label{fig:cn_Delaunay}\medskip
  \end{minipage}
  \end{figure}
\section{Uniform Setting}
\label{sc:linearbox-splinesStability}
Throughout this section, we assume that the vertices coincide with the sites of a lattice $\Lambda$. This is natural in some applications, such as image processing \cite{kimBoxSplineReconstruction2008}, but can also be a sensible choice in low dimensional learning problems \cite{Campos2021}. While a uniform grid constrains the model and thereby reduces its expressivity, it significantly improves the computational performance. In this setting, the parametrization of CPWL functions is naturally handled if one chooses shifts of a single linear box spline $B$ for the hat basis functions. The generated space $\mathrm{CPWL}(\mathcal{S})= \{\sum_{\M k \in \Lambda}c_{\M k}B(\cdot - \M k)\colon c_{\M k} \in \mathbb{C}\}$ is now shift-invariant and lends itself to the powerful tools of Fourier analysis \cite{DeBoor1994}. Note that, for any $\M x\inR^d$, the sum $\sum_{\M k \in \Lambda}c_{\M k}B(\M x - \M k)$ has at most $(d+1)$ nonzero arguments. This follows from the short support of linear box splines (we provide more details on the properties of linear box splines in Section \ref{subsec:boxsplines}).

\subsection{Linear Box Splines}
\label{subsec:boxsplines}
Box splines of any degree have been extensively studied. We refer the reader to the book by de Boor {\it et al.}  \cite{de1993box} for a general theory and a more comprehensive account. In this paper, we shall concentrate solely on linear box splines, which will in return allow us to derive precise specific results.

Consider a matrix $\V \Xi = [\V \xi_1\cdots\V \xi_d]\inR^{d\times d}$, where $(\V \xi_1,\ldots,\V \xi_d)$ is a collection of linearly independent vectors of $\R^d$. The matrix $\V \Xi$ generates a lattice of $\R^d$ whose sites are $\V \Xi \Z^d = \{\V \Xi \M k\colon \M k\inZ^d\}$. Moreover, let $\V \Xi_{d+1} =[\V \xi_1 \ldots\V \xi_{d+1}]\inR^{d\times(d+1)}$ with $\V \xi_{d+1}=\sum_{k=1}^d \V \xi_k$. The linear box spline $B_{\V \Xi_{d+1}}\colon\R^d\rightarrow\R$ generated by the collection of vectors $(\V \xi_1,\cdots,\V \xi_{d+1})$ is best defined its Fourier transform
\begin{equation}
  \label{eq:fourierlinearboxspline}
  \widehat{B}_{\V \Xi_{d+1}}(\V \omega) = |\det \V \Xi| \prod_{k=1}^{d+1} \frac{1-\ee^{-\ii \V \xi_{k}^T\V \omega}}{\ii\V \xi_{k}^T\V \omega}.
\end{equation}
The normalization factor $|\det \V \Xi|$ ensures consistency with our definition of a hat function, but is often not included in the literature. The fact that $B_{\V \Xi_{d+1}}$ is a CPWL function is made explicit with Proposition \ref{pr:greenexpansion}.
\begin{proposition}
  \label{pr:greenexpansion}
  \begin{equation}
    B_{\V \Xi_{d+1}}(\M x) =\sum_{\V \epsilon\in\{0,1\}^{d+1}}(-1)^{|\V\epsilon|} \min\left(\V \Xi^{-1}(\M x - \V \Xi_{d+1}\V \epsilon)\right)_+.
  \end{equation}
\end{proposition}
\begin{proof}
The product in \eqref{eq:fourierlinearboxspline} can be expanded as
\begin{equation}
  \label{eq:greenexpansionfourier}
  \widehat{B}_{\V \Xi_{d+1}}(\V \omega) = |\det \V \Xi| \left (\prod_{k=1}^{d+1} \frac{1}{\ii\V \xi_{k}^T\V \omega}\right )\times \sum_{\V \epsilon\in\{0,1\}^{d+1}}(-1)^{|\V\epsilon|}\ee^{-\ii (\V \Xi_{d+1} \V \epsilon)^T \V \omega}.
\end{equation}
By invoking Lemma \ref{lm:fourierminplus} (Appendix \ref{ap:green}) and making use of the general Fourier stretch theorem, we get that
\begin{equation}
  \min(\V \Xi^{-1}\M x)_+ \stackrel{\mathcal{F}}{\mapsto}|\det \V \Xi| \prod_{k=1}^{d+1} \left (\frac{1}{\ii\V \xi_{k}^T\V \omega} + \pi \delta(\V \xi_{k}^T\V \omega)\right ),
\end{equation}
where $\M x$ is the space variable, $\V \omega$ is the pulsation variable and $\delta$ is the Dirac distribution. Knowing that $(1-\ee^{-\ii \V \xi_{k}^T\V \omega})\delta(\V \xi_{k}^T\V \omega)=0$, we observe that \eqref{eq:fourierlinearboxspline} has the equivalent form
\begin{equation}
  \label{eq:toindefourier}
  \widehat{B}_{\V \Xi_{d+1}}(\V \omega) = |\det \V \Xi | \left (\prod_{k=1}^{d+1} \frac{1}{\ii\V \xi_{k}^T\V \omega}+ \pi \delta(\V \xi_{k}^T\V \omega)\right ) \sum_{\V \epsilon\in\{0,1\}^{d+1}}(-1)^{|\V\epsilon|}\ee^{-\ii \V \epsilon^T \V \Xi_{d+1}^T \V \omega}.
\end{equation}
We conclude by taking the inverse Fourier transform on both sides of \eqref{eq:toindefourier}.
\end{proof}
Proposition \ref{pr:greenexpansion} is illustrated in dimension $d=2$ in Figure \ref{fig:GHH_to_Hat}. This expansion gives a way to prove that the GHH model \cite{Wang2005} can represent any linear box spline. More interestingly, it provides a precise relation between the local and nonlocal representations, with explicit generalized hinging hyperplanes (namely, the shifts of a single generalized hinging hyperplane).
Our result is close to \cite{condatThreedirectionalBoxsplinesCharacterization2006}, where a similar formula is proven for three directional box splines of any degree, but only in dimension $d=2$. In any dimension and for box splines of any degree, a comparable decomposition is given in \cite{Horacsek2018EvaluatingBS}. Nonetheless, when applied to linear box splines, the expansion in \cite{Horacsek2018EvaluatingBS} is made of discontinuous Green functions and is less compact.
\begin{figure}[h!]
  \begin{minipage}{1.0\linewidth}
    \centering
    \centerline{\includegraphics[width=120mm]{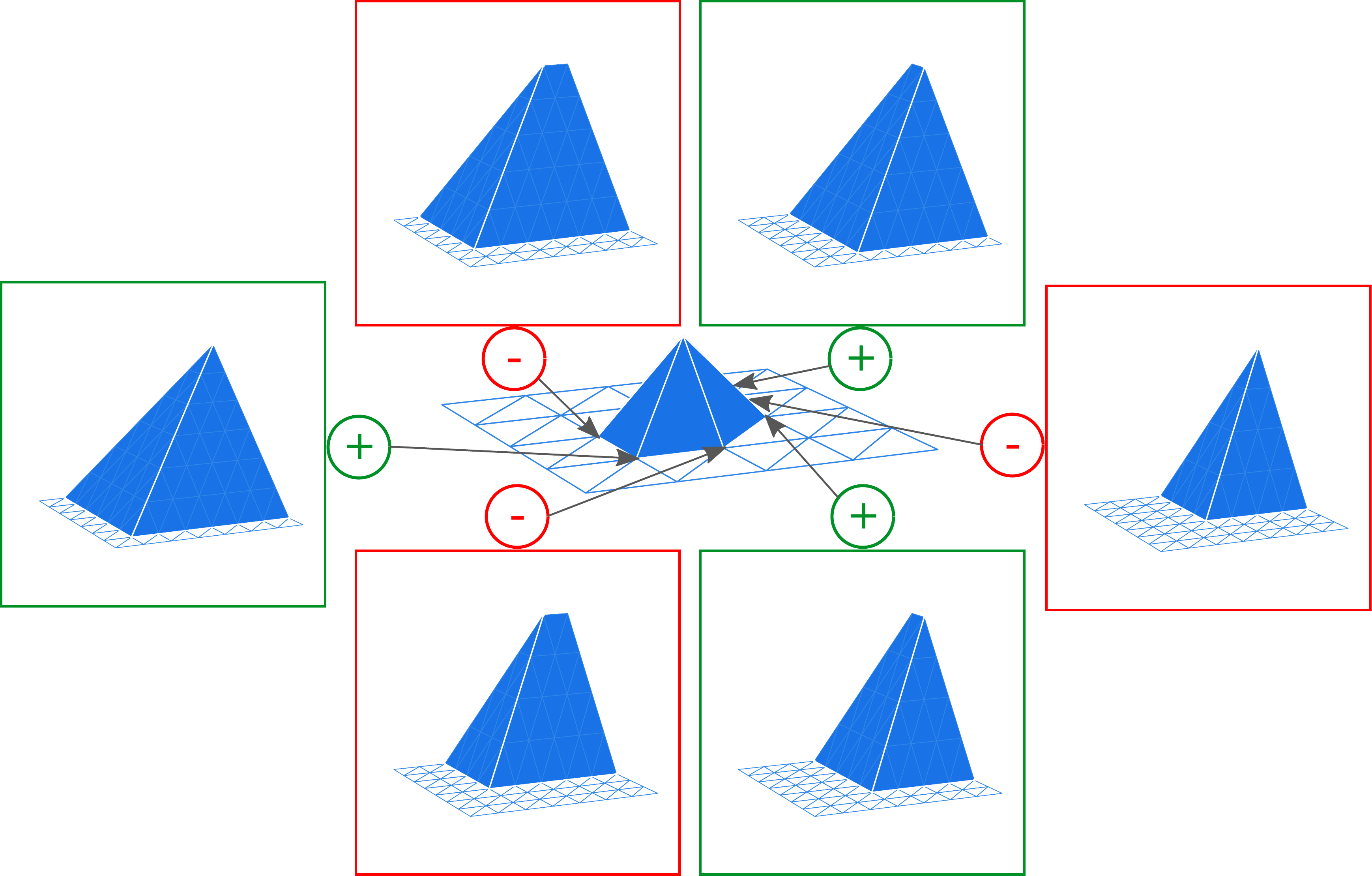}}%  \vspace{2.0cm}
    \caption{Linear decomposition of the 2D linear box spline in translates of the causal hinge functions $(x,y)\mapsto \min(x,y)_+$.}
    \label{fig:GHH_to_Hat}\medskip
  \end{minipage}
  \end{figure}

The key properties of the linear box spline $B_{\V \Xi_{d+1}}$ are
\begin{itemize}
  \item continuity and piecewise linearity (obvious from Proposition \ref{pr:greenexpansion} since $\M x\mapsto \min(\M x)_+$ is CPWL);
  \item compact support since the support is the Minkowski sum of $(\V \xi_1,\ldots,\V \xi_{d+1})$, which is a zonotope that is symmetric with respect to its center $\V \xi_{d+1}=\sum_{k=1}^d\V \xi_k$;
  \item $B_{\V \Xi_{d+1}}(\V \xi_{d+1}) = 1$ and $B_{\V \Xi_{d+1}}(\V \Xi \M k) = 0$ for $\M k \in \Z^d\backslash\{\V 1\}$;
  \item approximation power of order $2$ \cite{de1993box}, which means that the reconstruction error of sufficiently smooth decaying functions decreases with the square of the grid size.
\end{itemize}

\subsection{Derivation of the Exact Riesz-Basis Bounds}
Prior to giving the general Riesz-basis bounds of linear box splines (Theorem \ref{th:boxsplinesRB}), we present a simple but powerful result.
  \begin{proposition}
    \label{prop:RBInvarianceLattice}
    Let $(\V \xi_1,\cdots,\V \xi_{d})$ be a family of linearly independent vectors of $\R^d$, $\V \xi_{d+1}=\sum_{k=1}^d \V \xi_k$, $\V \Xi_{d+1} = [\V \xi_1\cdots\V \xi_{d+1}]$ and $\M U\inR^{d\times d}$ an invertible matrix. Then, the Riesz bounds of $B_{\M U\V \Xi_{d+1}}$ are the ones of $B_{\V \Xi_{d+1}}$ scaled by $\sqrt{|\det{(\M U)}|}$.
  \end{proposition}
  \begin{proof}
  From Proposition \ref{pr:greenexpansion}, we infer that $B_{\M U \V \Xi_{d+1}} = B_{\V \Xi_{d+1}}\circ \M U^{-1}$. (The generalized version of this relation $M_{\M U \V \Xi } = \abs{\det \M U^{-1}} M_{\V \Xi } \circ \M U^{-1}$ is given in \cite{de1993box} and holds true for unnormalized box splines of any degree.) In addition, for any $f \in L_2(\R^d)$, a change of variable allows one to show that $\|f\circ \M U^{-1}\|_{L_2}^2 =  \abs{\det(\M U)}\|f\|_{L_2}^2$, which means that, for $c\in \ell_2(\mathbb{Z})$, it holds that
  \begin{equation}
  \|\sum_{\M k\inZ^d}c_{\M k} B_{\M U \V \Xi_{d+1}}(\cdot-\M k)\|_{L_2}^2 = \| ((\sum_{\M k\inZ^d}c_{\M k} B_{\V \Xi_{d+1}})\circ \M U^{-1})(\cdot-\M k)\|_{L_2}^2 = \abs{\det(\M U)}\|\sum_{\M k\inZ^d}c_{\M k} B_{\V \Xi_{d+1}}(\cdot-\M k)\|_{L_2}^2,
  \end{equation}
 which allows us to conclude.
  \end{proof}

\begin{theorem}
  \label{th:boxsplinesRB}
  Let $(\V \xi_1,\ldots,\V \xi_{d})$ be a family of $d$ free vectors of $\R^d$, $\V \xi_{d+1}=\sum_{k=1}^d \V \xi_k$, and $\V \Xi_{d+1} = [\V \xi_1\cdots\V \xi_{d+1}]$. Then, the collection $\{B_{\V \Xi_{d+1}}(\cdot-\V \Xi \M k)\colon \M k \in \Z^d\}$ of linear box splines forms a Riesz basis with bounds
  \begin{equation}
    A = \sqrt{\frac{|\det \V \Xi|}{(d+2)}} \;\text{  and  }\; B = \sqrt{|\det \V \Xi|}.
  \end{equation}
The associated Riesz condition number is $\sqrt{d+2}$.
  \end{theorem}
  \begin{proof}
  Following Proposition \ref{prop:RBInvarianceLattice}, we focus solely on the cartesian lattice and denote the Cartesian linear box spline by $D$. In this case, it is known that the simplices form a Kuhn/Freudenthal triangulation \cite{Kim2010}. The vertices of any simplex $s = \mathrm{conv}(\M v_0,\ldots,\M v_d)$ of the triangulation take the form
  \begin{equation}
    \label{eq:kuhn}
  \M v_k = \M v_0 + \sum_{p = 1}^k \M e_{\V\sigma(p)},
  \end{equation}
  where $\V\sigma$ is a permutation of the set $\{0,\ldots,d\}$ and $k=1,\ldots,d$. Correspondingly, for $p=0,\ldots, d$, there is a unique vertex $\M v_{k_p}$ in each simplex  such that $\abs{\M v_{k_p}} \equiv p \mod (d+1)$. Indeed, \eqref{eq:kuhn} yields that $|\M v_k|= |\M v_0| + k$. Then, we use the Fourier characterization \eqref{eq:RBFourier} to find the Riesz bounds. Specifically, we have to find the essential extrema of
  \begin{align}
    \widehat g\colon \V \omega \mapsto \sum_{\M k \inZ^d}\dotprod{D}{D(\cdot - \M k)}\ee^{-\ii \V \omega^T \M k}.
  \end{align}
Since the basis function $D$ is compactly supported, the sum is finite and the function $\widehat g\colon\R^d\rightarrow \C$ is continuous and $2\pi$-periodic with respect to each coordinate. We can therefore simply look for the maximum and the minimum of $\widehat g$ in the hypercube $[0,2\pi)^d$. We apply the triangular inequality, use the non negativity of $D$ and invoke the partition of unity of the linear box spline to conclude that the maximum of $\widehat g$ is attained for $\V \omega = \V 0$. With \eqref{eq:L1norm} we find the upper bound
  $B^2 = \frac{\vol{\mathrm{St}(D)}}{d+1}$, where $\mathrm{St}(D)$ is the support of $D$ or, equivalently, the star of the vertex located at $\V 1$. Now, for the minimum, we evaluate $\widehat g$ at $\V \omega_0 = (-\frac{2\pi}{d+1}\V 1)$ and get
  \begin{align}
  \sum_{\M k \inZ^d}\dotprod{D}{D(\cdot - \M k)}\ee^{-\ii \V {\omega_0}^T \M k}&=\dotprod{D}{\sum_{\M k \inZ^d}\zeta_{d+1}^{|\M k|}{D(\cdot - \M k)}}\nonumber\\
  &= \sum_{s\in \mathrm{St}(D)}\dotprod{D}{\sum_{\M k \inZ^d}\zeta_{d+1}^{|\M k|} D(\cdot - \M k)}_s\nonumber\\
  &= \sum_{s\in \mathrm{St}(D)}\dotprod{D}{\sum_{(\M k -\V 1)\inZ^d\cap s}\zeta_{d+1}^{|\M k|} D(\cdot - \M k)}_s,
  \end{align}
  where $\dotprod{f}{h}_s = \int_{\M x\in s}\overline{f(\M x)}h(\M x)\dint \M x$. On one hand, we apply Lemma \ref{lm:innerprodcutsimplexcomplex} and use the inherent structure of the Kuhn triangulation displayed in \eqref{eq:kuhn} to deduce that
  \begin{align}
  \min_{\V \omega\in[0,2\pi]^d}\widehat g(\V \omega)\leq \widehat g(\V \omega^*)&=\sum_{s\in St_{D}} \frac{1}{(d+1)(d+2)}\vol{s} \begin{bmatrix}
  1 &0& \ldots &0
  \end{bmatrix}  \begin{bmatrix}
  2 &1 &\ldots&1\\
  1 &\ddots & \ddots & \vdots\\
  \vdots& \ddots & \ddots & 1\\
  1 & \cdots & 1 & 2
  \end{bmatrix}   \begin{bmatrix}
  1 \\ \zeta_{d+1} \\\vdots \\ \zeta_{d+1} ^d
  \end{bmatrix}\nonumber\\
  &= \frac{\vol{\mathrm{St}(D)}}{(d+1)(d+2)}.
  \end{align}
  On the other hand, Theorem \ref{th:RBIrregular} implies that $\min_{\V \omega\in[0,2\pi]^d}\widehat g(\V \omega)\geq \frac{\vol{\mathrm{St}(D)}}{(d+1)(d+2)}$, from which we infer that $A^2=\frac{\vol{\mathrm{St}(D)}}{(d+1)(d+2)}$.
  Moreover, we have that $\int_{\R^d}D(\M x)\dint \M x = \widehat{D}(\V 0) = \det {(\M I)}= 1$ (from \eqref{eq:fourierlinearboxspline}) and, since the box spline $D$ is nonnegative, $\int_{\R^d}D(\M x)\dint \M x=\|D\|_{L_1}=\frac{\vol{\mathrm{St}(D)}}{(d+1)}$ (from \eqref{eq:L1norm}). In short, $\vol{\mathrm{St}(D)}=(d+1)$, which allows us to derive the result for the Cartesian lattice. The extension to any lattice then follows from Proposition \ref{prop:RBInvarianceLattice}.
  
  \end{proof}
  
  Theorem \ref{th:RBIrregular} and \ref{th:boxsplinesRB} yield the same bounds for linear box splines, which confirms the good quality of the bounds provided for irregular triangulations.
  
  In the proof of Theorem \ref{th:boxsplinesRB} we showed, on one hand, that the volume of the star of a vertex in the Kuhn triangulation is $(d+1)$. On the other hand, the volume of the simplices of this triangulation can be readily computed and amounts to $\frac{1}{d!}$. We deduce that the linear box spline is made of $(d+1)!$ nonzero affine pieces.

  From Theorem \ref{th:boxsplinesRB}, the Riesz condition number of the linear box-spline parametrization is $\sqrt{d+2}$, which grows with the dimension. However, this metric only reflects extreme cases. A good estimate of the average behavior of the parametrization is given by the mean of the function $\widehat g$ (defined in the proof of Theorem \ref{th:boxsplinesRB}) over $[0,2\pi]^d$, compute as
    \begin{align}
      \frac{1}{(2\pi)^d}\int_{\V \omega \in [0,2\pi]^d}\sum_{\M k \inZ^d}\dotprod{D}{D(\cdot - \M k)}\ee^{-\ii \V \omega^T \M k} \dint \V \omega&= \dotprod{D}{D} = \|D\|_{L_2}^2=2 \frac{\vol{(\mathrm{St}(D)}}{(d+2)(d+1)}=2 \min_{\V \omega \in [0,2\pi]^d}\widehat g(\V \omega).
      \end{align}
  We infer that $\widehat g$ rarely takes values close to its upper bound, especially in high dimensions, which was observed for affine functions on a simplex in Lemma \ref{lm:conditionnumberdistribution}. In short, although the condition number scales badly with the dimension, most linear combinations of box splines will behave closely to the lower bound and the dimension should not be significantly detrimental to the stability of the parametrization.

\section{Conclusion}
We have provided a precise measure of the stability of the parametrization of CPWL functions with hat basis functions on any triangulation. First, we have estimated the $\ell_2\rightarrow L_2$ condition number of the parametrization in full generality. We have found that it is mainly determined by the relative volume of the star of the vertices of the triangulation, namely, the relative size of the support of the hat functions. Then, we have proposed a method to compute the condition number for finite triangulations. It boils down to the computation of the condition number of a given matrix. When the vertices lie at the sites of a lattice, we parametrize the CPWL functions with linear box splines. We have provided a formula to relate these local basis functions to nonlocal ReLU-like functions. In this uniform setting, we have proved that the condition number only depends on the dimension $d$ and is $\sqrt{d+2}$. Although it increases with the dimension, we noticed that, from a stochastic point of view, dimension will only rarely affect the stability of the local representation.
\newpage
\appendix
\section{Condition Number of the Nonlocal Parametrization}
\label{ap:condition}
Consider the nonlocal parametrization of one-dimensional CPWL functions with control points $v_1<\cdots<v_{K}\inR$
\begin{equation}
  T\{\V \theta\}(x) =  \theta_1 + \theta_2 (x-v_1)+\sum_{k=2}^{K-1} \theta_{k+1}(x-v_k)_+.
\end{equation}
Given target values $y_1,\ldots,y_{K} \inR$, the parameters $\theta_1,\ldots,\theta_{K}\inR$ have to satisfy for $p=1,\ldots,K$ that
\begin{align}
  y_p = T\{\V \theta\}(v_p) = \theta_1 + \theta_2 (v_p-v_1)+\sum_{k=2}^{K-1} \theta_{k+1}(v_p-v_k)_+,
\end{align}
which, assuming a constant step size $h=(v_{k+1}-v_k)$, further simplifies in
$
  y_p = \theta_1 + h \sum_{k=1}^{p-1} \theta_{k+1}(p-k)
$.
This yields the matrix equation
\begin{equation}
  \V \theta = \frac{1}{h}\begin{bmatrix}
    1/h &0  & \cdots &  \cdots & 0\\ 
    1/h & 1 & 0 &\cdots&0 \\ 
     \vdots& 2 & \ddots & \ddots &\vdots\\ 
    \vdots & \vdots & \ddots  & \ddots&0\\ 
    1/h & (K-1) & \cdots &2 &1
    \end{bmatrix}^{-1}\M y = \M M^{-1}\M y.
\end{equation}
To give a lower bound to the condition number of the problem, we remark that $\M M \M e_2 = h \sum_{k=1}^{K-1}k \M e_{k+1} \text{ and } \M M \M e_K = h \M e_K$, where $\M e_k$ are the canonical vectors of $\R^d$. We infer that $\|\M M^{-1}(\sum_{k=1}^{K-1}k \M e_k)\|_2/\|\sum_{k=1}^{K-1}k \M e_k)\|_2=(K(K-1)(2K-1)/6)^{-1/2}/h$ and that $\frac{\|\M M^{-1}\M e_K\|_2}{\|\M e_K\|_2}= 1/h$. It implies that the $\ell_2$ condition number $r$ of the problem satisfies
\begin{equation}
  r=\max _{\M a,\M b\inR^d\backslash\{\V 0\}}\left\{{\frac {\left\|\M M^{-1}\M a\right\|}{\|\M a\|}}{\frac {\|\M b\|}{\left\|\M M^{-1}\M b\right\|}}\right\}\geq \sqrt{\frac{K(K-1)(2K-1)}{6}}.
\end{equation}
\section{The Generalized Hinging Hyperplane Generating Function of Linear Box Splines}
\label{ap:green}
The Fourier transform $\widehat{h}$ of the Heaviside function $h\colon x \mapsto\begin{cases}1, & x\geq 0\\0,\text{ otherwise},\end{cases}$ is given by
\begin{equation}
  \widehat{h}\colon \omega\mapsto\frac{1}{\ii\omega} + \pi \delta(\omega).
\end{equation}
From this we infer $\widehat{H}$, the Fourier transform of $H\colon \M x\mapsto \prod_{k=1}^d h(x_k)$, the separable version of $h$ in $d$ dimensions, 
\begin{equation}
  \widehat{H} \colon \V \omega \mapsto \prod_{k=1}^{d} \left(\frac{1}{\ii\omega_k} +\pi \delta(\omega_k)\right).
\end{equation}

We define a directional Heaviside wall function as $W\colon \M x \mapsto h(x_d)\prod_{k=1}^{d-1} \delta(x_k)$ and the matrix $\M A_d = \begin{bmatrix}
  1&-1\\
  &\ddots&\ddots\\
  & & 1&-1\\
  1&\cdots &\cdots&1
\end{bmatrix}$.
\begin{lemma}
  \label{lm:convequation}
  \begin{equation}
    \forall \M x \inR^d, (H*(W\circ \M A_d))(\M x) = \min(\M x)_+.
  \end{equation}
\end{lemma}
\begin{proof}
  We have that
  \begin{align}
    (H*(W\circ \M A_d))(\M x) = \int_{\R^d}H(\M x - \M t)W(\M A_d \M t)\dint \M t.
  \end{align}
 In the sequel, we take advantage of two properties.
 \begin{itemize}
   \item We prove that $\det (\M A_d) = d$ by induction using the Laplace/cofactor expansion along the last column.
   \item We observe that $\M A_d  \V 1 = d \M e_d$, which in turn implies that $\M A_d^{-1}  \M e_d = d^{-1}  \V 1$.
 \end{itemize}
 We now use the change of variable $\M y= \M A_d\M t$. It yields that
  \begin{align}
    (H*(W\circ \M A_d))(\M x) &= (1/d)\int_{\R^d} H(\M x - \M A_d^{-1}\M y) h(y_{d})\prod_{k=1}^{d-1}\delta(y_k)\dint \M y \nonumber\\
    &= (1/d)\int_{\R^d} H(\M x - y_d \M A_d^{-1}\M e_d) h(y_{d})\prod_{k=1}^{d-1}\delta(y_k)\dint \M y\nonumber\\
    &= (1/d)\int_{\R^d} H(\M x - \frac{y_d}{d} \V 1) h(y_{d})\prod_{k=1}^{d-1}\delta(y_k)\dint \M y\nonumber\\
    &=(1/d)\int_{\R^d} h(y_{d})\left (\prod_{k=1}^{d} h(x_k - y_d/d)\right ) \left (\prod_{k=1}^{d-1}\delta(y_k)\right )\dint \M y\nonumber\\
    &=(1/d)\int_{\R} h(y_{d})\left (\prod_{k=1}^{d} h(x_k - y_d/d)\right )\dint y_d.
  \end{align}
  The quantity $\prod_{k=1}^{d} h(x_k - y_d/d) h(y_{d})$ is nonzero when $y_d>0$ and $y_d<dx_k$ for $k=1,\ldots,d$, which is equivalent to $y_d>0$ and $y_d<d\min(x_k)$. We can now conclude that $(H*(L\circ \M A_d))(\M x) = \min(\M x)_+$.
\end{proof}
\begin{lemma}
  \label{lm:fourierminplus}
  \begin{equation}
    \min(\M x)_+\stackrel{\mathcal{F}}{\mapsto}\prod_{k=1}^{d+1} \left(\frac{1}{\ii\omega_k} +\pi \delta(\omega_k)\right),
  \end{equation}
  where $\omega_{d+1}=\sum_{k=1}^d \omega_k$.
\end{lemma}
\begin{proof}
  From Lemma \ref{lm:convequation}, we have that $(H*(W\circ \M A_d))(\M x) = \min(\M x)_+$. The function $W$ is separable and its Fourier transform reads $\widehat{W}(\V \omega) = \frac{1}{\ii \omega_d}+\pi \delta(\omega_d)$. In addition, the general stretch theorem implies that
  \begin{equation}
    (W\circ \M A_d) \stackrel{\mathcal{F}}{\mapsto}\frac{1}{d}\widehat{W}(\M A_d^{-T}\V \omega) = \frac{1}{d}\left(\frac{1}{\ii \M e_d^T \M A_d^{-T}\V \omega}+\pi \delta(\M e_d^T \M A_d^{-T}\V \omega)\right).
  \end{equation}
  Now, we use that $\M A_d^{-1} \M e_d = d^{-1}\M e_d$ and the effect of a dilation on the Dirac distribution to conclude that
  \begin{equation}
    (W\circ \M A_d) \stackrel{\mathcal{F}}{\mapsto}\frac{1}{d}\widehat{W}(\M A_d^{-T}\V \omega) = \left(\frac{1}{\ii \omega_{d+1}}+\pi \delta( \omega_{d+1})\right).
  \end{equation}
 We reach the conclusion by using the Fourier transform of $H$ and by transforming the convolution into a product in the Fourier domain.
\end{proof}
\newpage
\bibliography{All3}
\bibliographystyle{elsarticle-num}
\end{document}